\documentclass[11pt,a4paper]{amsart}

\usepackage[T1]{fontenc}
\usepackage[utf8]{inputenc}
\usepackage{lmodern}
\usepackage{amsmath,amssymb,amsthm,mathtools}
\usepackage{microtype}
\usepackage{enumitem}
\usepackage{xcolor}
\usepackage{tikz}
\usepackage{hyperref}
\hypersetup{hidelinks}
\theoremstyle{plain}
\newtheorem{theorem}{Theorem}[section]
\newtheorem{proposition}[theorem]{Proposition}
\newtheorem{lemma}[theorem]{Lemma}
\newtheorem{corollary}[theorem]{Corollary}

\theoremstyle{definition}
\newtheorem{remark}[theorem]{Remark}

\newtheorem{question}[theorem]{Question}

\newcommand{\R}{\mathbb{R}}

\newcommand{\Q}{\mathbb{Q}}
\newcommand{\OO}{\mathcal{O}}
\newcommand{\CC}{\mathcal{C}}
\newcommand{\Spec}{\operatorname{Spec}}
\newcommand{\dist}{\operatorname{dist}}
\newcommand{\cc}{\operatorname{cc}}
\newcommand{\dhc}{d_{H^c}}

\title[Generic failure of uniform spectral separation]
{Generic failure of uniform separation in planar Dirichlet spectra}
\author[V. Boulard]{Vincent Boulard}
\address{CERMICS, CNRS, \'Ecole nationale des ponts et chauss\'ees, Institut Polytechnique de Paris, Marne-la-Vall\'ee, France}
\email{vincent.boulard@enpc.fr}

\begin{document}

\begin{abstract}
{Can a bounded planar domain have a simple Dirichlet spectrum with
uniformly separated consecutive eigenvalues? Dimension two is critical:
Weyl's law permits both uniform separation and arbitrarily small gaps. We
prove that uniform separation is nevertheless exceptional in a natural
rough-domain setting, even after multiplicities are removed.}
{Let $D\subset\R^2$ be a bounded domain and, for $\ell\geq1$, let $\CC_\ell(D)$ be the
space of nonempty connected open sets $\Omega\subset D$ such that
$\overline D\setminus\Omega$ has at most $\ell$ connected components,
endowed with the complementary-Hausdorff topology. We prove that
$\CC_\ell(D)$ is completely metrizable and Baire, and that smooth domains
are dense in it. If
$0<\nu_1(\Omega)<\nu_2(\Omega)<\cdots$ are the distinct Dirichlet
eigenvalues, our main result states that}
\[
{\left\{\Omega\in\CC_\ell(D):
\inf_{m\geq1}\bigl(\nu_{m+1}(\Omega)-\nu_m(\Omega)\bigr)=0\right\}}
\]
{is residual. This statement requires no simplicity assumption.
Combining it with our transfer of Micheletti's classical generic-simplicity
theorem to $\CC_\ell(D)$ shows that a generic domain has simple spectrum and
consecutive gaps with zero lower limit. The proof uses \v{S}ver\'ak's planar
spectral continuity theorem and a local surgery that implants an arbitrarily
high pair of close consecutive distinct eigenvalues.}
\end{abstract}

\maketitle

\section{Introduction}\label{sec:intro}

{Let $\Omega\subset\R^d$ be a bounded domain and write}
\[
 0<\lambda_1(\Omega)\leq\lambda_2(\Omega)\leq\cdots
\]
{for its Dirichlet eigenvalues, repeated according to multiplicity.
{The uniform-gap condition for the full spectrum is}}
\begin{equation}\label{eq:full-uniform-gap}
 {\inf_{j\geq1}
 \bigl(\lambda_{j+1}(\Omega)-\lambda_j(\Omega)\bigr)>0.}
\end{equation}
{Since multiplicities make this condition fail trivially, let}
\[
 0<\nu_1(\Omega)<\nu_2(\Omega)<\cdots
\]
{enumerate the distinct spectral values. They are uniformly separated if}
\begin{equation}\label{eq:distinct-uniform-gap}
 {\inf_{m\geq1}
 \bigl(\nu_{m+1}(\Omega)-\nu_m(\Omega)\bigr)>0.}
\end{equation}
{We prove that \eqref{eq:distinct-uniform-gap}, and hence
\eqref{eq:full-uniform-gap}, fails generically.}

{The Weyl law singles out dimension two. In dimension $d$, it gives}
\begin{equation}\label{eq:weyl}
 {\lambda_j(\Omega)\sim 4\pi^2
 \left(\frac{j}{\omega_d|\Omega|}\right)^{2/d},
 \qquad j\longrightarrow\infty,}
\end{equation}
{where $\omega_d$ is the unit-ball volume \cite{Weyl1912}. On an
interval the gaps grow linearly. If $d\geq3$ and the spectrum is simple, the
average}
$\frac1j\sum_{k=1}^{j}\bigl(\lambda_{k+1}(\Omega)-\lambda_k(\Omega)\bigr)
=\bigl(\lambda_{j+1}(\Omega)-\lambda_1(\Omega)\bigr)/j$ tends to zero by
\eqref{eq:weyl}, {so}
\[
\inf_{m\geq1}\bigl(\nu_{m+1}(\Omega)-\nu_m(\Omega)\bigr)=0.
\]
{In dimension two, however,}
\begin{equation*}
 {\lambda_j(\Omega)\sim\frac{4\pi}{|\Omega|}j.}
\end{equation*}
{The mean spacing has a positive finite scale, so the Weyl law permits
both uniform separation and arbitrarily small gaps. Thus $d=2$ is critical
for this problem in the simple-spectrum regime.}

{Rectangles illustrate both the difficulty and the distinction between
the two conditions. For $\Omega=(0,a)\times(0,b)$, the eigenvalues are}
\[
 {\pi^2\left(\frac{m^2}{a^2}+\frac{n^2}{b^2}\right),}
 \qquad m,n\geq1.
\]
{If $a^2/b^2\in\mathbb Q$, the distinct eigenvalues lie in a fixed
lattice, so \eqref{eq:distinct-uniform-gap} holds but multiplicities rule out
\eqref{eq:full-uniform-gap}. If $a^2/b^2$ is irrational, the spectrum is
simple and its smallest gap among the first $N$ eigenvalues is
$O(N^{-1/2})$ \cite[Proposition~2.2]{BlomerBourgainRadziwillRudnick2016}.
The essentially Poissonian $N^{-1}$ scale, predicted by Berry--Tabor
statistics \cite{BerryTabor1977}, requires additional arithmetic hypotheses
in that work. Thus no rectangle satisfies \eqref{eq:full-uniform-gap}. Our
novelty is the Baire genericity of small gaps in an infinite-dimensional
space of rough domains, not their first occurrence.}

{This high-frequency question differs from the fundamental-gap problem.
For convex domains, Andrews and Clutterbuck proved the sharp bound
$\lambda_2-\lambda_1\geq3\pi^2/\operatorname{diam}(\Omega)^2$
\cite{AndrewsClutterbuck2011}, with no analogous geometric bound available
uniformly at high frequency. Payne--P\'olya--Weinberger estimates instead
control consecutive eigenvalue ratios \cite{PaynePolyaWeinberger1956}, not
uniform additive gaps. Such gaps underlie Ingham inequalities
\cite{Ingham1936} and many control arguments for conservative systems. For
parabolic equations they enter the moment method through biorthogonal
families for $(e^{-\lambda_jt})_j$, beginning with Fattorini and Russell
\cite{FattoriniRussell1971}. Cannarsa, Martinez and Vancostenoble distinguish
uniform from asymptotic gaps \cite{CannarsaMartinezVancostenoble2020}.}

{Uniform spectral separation also appears in Schr\"odinger control. In
their study of linearized dipole-controlled equations, Beauchard, Chitour,
Kateb and Long used it in trigonometric moment methods and asked whether a
{planar domain with $C^1$ boundary} can satisfy
\eqref{eq:full-uniform-gap} \cite{BeauchardChitourKatebLong2009}. See also
Zuazua \cite{Zuazua2003} on Ingham inequalities and spectral gaps in
Schr\"odinger controllability. The present theorem does not settle this
regular-domain existence question, but shows that uniform separation is
exceptional in the complementary-Hausdorff spaces introduced below.}

{We seek the stronger Baire-category conclusion that even the distinct
gaps have zero infimum. Its proof requires closing thin walls and chambers,
which is not a small perturbation in smooth domain topologies such as
Micheletti's \cite{Micheletti1972,MichelettiMetric1972}. We therefore use
\v{S}ver\'ak's complementary-Hausdorff framework. Fix a bounded domain
$D\subset\R^2$ and, for $\ell\geq1$, set}
\begin{equation*}
 \OO_\ell(D):=
 \bigl\{\Omega\subset D\text{ open}:\overline D\setminus\Omega
 \text{ has at most }\ell\text{ connected components}\bigr\},
\end{equation*}
including the empty set, and the subspace of connected domains
\begin{equation*}
 \CC_\ell(D):=
 \bigl\{\Omega\in\OO_\ell(D):\Omega\neq\varnothing
 \text{ and }\Omega\text{ is connected}\bigr\}.
\end{equation*}
For $\Omega\in\OO_\ell(D)$, set $K_\Omega:=\overline D\setminus\Omega$ and define
\begin{equation}\label{eq:dhc}
 \dhc(\Omega_1,\Omega_2):=d_H(K_{\Omega_1},K_{\Omega_2}).
\end{equation}
{The space $\OO_\ell(D)$ is compact. In the plane, complementary-Hausdorff
convergence in this class implies $\gamma$-convergence and hence convergence
of every Dirichlet eigenvalue \cite{Sverak1993}. See Bucur--Buttazzo
\cite{BucurButtazzo2005} and Henrot--Pierre \cite{HenrotPierre2018}, as well
as Bucur--Trebeschi \cite{BucurTrebeschi1998} for nonlinear extensions under
capacitary assumptions. Thus the topology accommodates closed walls while
preserving the spectrum. However, $\CC_\ell(D)$ is not closed in
$\OO_\ell(D)$ because a thin neck may close in the limit. Since residual
sets need not be dense outside Baire spaces, we first establish the required
topological framework.}

\begin{theorem}[Topological and spectral framework]\label{thm:framework-intro}
{Let $D\subset\R^2$ be a bounded domain and let $\ell\geq1$. Then the
following assertions hold.}
\begin{enumerate}[label=\textup{(\roman*)}]
 \item {The space $\CC_\ell(D)$ is a $G_\delta$ subset of the compact
 space $\OO_\ell(D)$. In particular, $\CC_\ell(D)$ is completely metrizable
 and Baire.}
 \item {Connected domains with $C^\infty$ boundary are dense in
 $\CC_\ell(D)$ for the complementary-Hausdorff topology.}
 \item {The set}
 \begin{equation}\label{eq:simple-set}
  {\mathcal S_\ell:=
  \bigl\{\Omega\in\CC_\ell(D):
  \lambda_j(\Omega)<\lambda_{j+1}(\Omega)
  \text{ for every }j\geq1\bigr\}}
 \end{equation}
 {is residual in $\CC_\ell(D)$.}
\end{enumerate}
\end{theorem}

{The three assertions have distinct statuses. Items~\textup{(i)} and
\textup{(ii)} are the new rough-topology contributions. They supply the
Baire setting and connect it to smooth perturbations. Micheletti's generic
simplicity theorem \cite{Micheletti1972,MichelettiMetric1972} and Teytel's
codimension-two analysis \cite{Teytel1999} concern smooth parameter spaces,
locally within a fixed diffeomorphism class. They do not provide the
$G_\delta$ description, smooth approximation, or transfer to
$\CC_\ell(D)$. This transfer is the new content of~\textup{(iii)}:
spectral continuity gives openness and~\textup{(ii)} gives density. Related
generic-simplicity results include Albert and Uhlenbeck
\cite{Albert1975,Uhlenbeck1976}, Henry \cite{Henry2005}, and Privat and
Sigalotti \cite{PrivatSigalotti2010}. We can now state the main result.}

\begin{theorem}[Generic failure of uniform separation]\label{thm:main}
{Let $D\subset\R^2$ be a bounded domain and let $\ell\geq1$. Then}
\[
 {\mathcal G_\ell:=
 \left\{\Omega\in\CC_\ell(D):
 \inf_{m\geq1}
 \bigl(\nu_{m+1}(\Omega)-\nu_m(\Omega)\bigr)=0\right\}}
\]
{is a dense $G_\delta$ set, hence is residual in $\CC_\ell(D)$ for the complementary-Hausdorff topology.}
\end{theorem}

{Neither the statement nor the proof assumes simplicity. We work with
the eigenvalues $\lambda_j$, counted with multiplicity, because each
fixed-index map is continuous. The inequality
$\lambda_j<\lambda_{j+1}$ marks exactly a gap between consecutive distinct
values. If $\mathcal U_{k,N}$ denotes the domains having such a gap smaller
than $1/k$ at some $j\geq N$, {Section~\ref{sec:baire} proves that}}
\[
{\mathcal G_\ell
=\bigcap_{k\geq1}\bigcap_{N\geq1}\mathcal U_{k,N}.}
\]
{It therefore suffices to prove that every $\mathcal U_{k,N}$ is open
and dense. Spectral continuity gives openness, and the following surgery
gives density.}

{In the local ``double-lock'' surgery, a finite tree attaches two
circular walls to one complementary component of an arbitrary
$\Omega\in\CC_\ell(D)$. Open doors keep the domain connected. Closing them
isolates two disks whose radii are tuned so that their first eigenvalues form
a distinct, consecutive, arbitrarily close pair at arbitrarily high index,
see Lemma~\ref{lem:attachment} and Proposition~\ref{prop:tuning}.
\v{S}ver\'ak's theorem transfers this positive small gap to the connected
open-door domains.}

\begin{corollary}[Simple-spectrum refinement]\label{cor:simple-refinement}
{The set}
\[
 {\left\{\Omega\in\mathcal S_\ell:
 {\inf_{j\geq1}}
 \bigl(\lambda_{j+1}(\Omega)-\lambda_j(\Omega)\bigr)=0
 \right\}}
\]
{is residual in $\CC_\ell(D)$. Thus the failure of uniform separation
persists generically after all multiplicities have been excluded.}
\end{corollary}

{Thus domains satisfying \eqref{eq:full-uniform-gap} are meagre in
$\CC_\ell(D)$. Theorem~\ref{thm:main} gives arbitrarily small distinct gaps
without simplicity, while intersecting $\mathcal G_\ell$ with the independent
generic-simplicity set gives Corollary~\ref{cor:simple-refinement}.}

{The argument also yields a finite-scale consequence for smooth domains.
For every fixed $k,N\geq1$, the openness and density of $\mathcal U_{k,N}$,
together with Theorem~\ref{thm:framework-intro}\textup{(ii)}, imply that
smooth domains in $\mathcal U_{k,N}$ are dense in $\CC_\ell(D)$. Thus every
complementary-Hausdorff neighborhood contains a smooth domain with a positive
gap smaller than $1/k$ at some index $j\geq N$. The smooth domain may depend
on $k$ and $N$, so this does not give a dense set of smooth domains violating
\eqref{eq:full-uniform-gap}. A Baire-generic conclusion in a Micheletti
$C^m$ topology would require
a different density mechanism. Indeed, the countable intersection defining
$\mathcal G_\ell$ cannot be recovered from smooth density alone, and the
chamber-forming perturbations, although $\dhc$-small, are not $C^m$-small for
$m\geq1$, even when constructed with smooth boundaries.}

{Several approximation results are related to
Theorem~\ref{thm:framework-intro}. Briani, Buttazzo and Prinari obtained
interior Lipschitz approximations of planar finite-perimeter sets, preserving
an upper bound on bounded complementary components and giving
complementary-Hausdorff convergence \cite{BrianiButtazzoPrinari2022}.
Schmidt proved strict smooth interior approximation when
$\mathcal H^1(\partial\Omega)=P(\Omega)$ \cite{Schmidt2015}, and Ball and
Zarnescu gave topology-preserving $C^\infty$ approximations of $C^0$ domains
\cite{BallZarnescu2017}. Proposition~\ref{prop:smooth-density}, tailored to
our category argument, assumes neither finite perimeter nor boundary
regularity, preserves connectedness and the bound on complementary
components, and gives quantitative $C^\infty$ approximation in
$\dhc$.}

{The surgery also recalls dumbbell domains and spectral stability under
domain perturbation \cite{RauchTaylor1975,Arrieta1995,Taylor2013}. Hillairet
and Judge used shrinking slits to prove generic simplicity for multiply
connected polygons \cite{HillairetJudge2010}. Here two nearly closed walls
create tunable Dirichlet chambers and place a close pair at arbitrarily high
index near every domain in $\CC_\ell(D)$. Colin de Verdi\`ere's tunnelling
construction \cite{ColinDeVerdiere1987} realizes any strictly increasing
finite list beginning at $0$ as the initial Neumann spectrum of a suitable
smooth simply connected planar domain, using chambers, narrow passages, and
a weighted graph Laplacian. Our construction is instead Dirichlet and local
near a prescribed domain, attaches the walls to its complement without
increasing the number of complementary components, and preserves
connectedness for every positive door width.}

{Section~\ref{sec:framework} proves the framework theorem,
Section~\ref{sec:surgery} constructs and tunes the double lock,
Section~\ref{sec:baire} proves the main theorem and corollary, and
Section~\ref{sec:questions} discusses the remaining exceptional uniform-gap
problem.}

\section{A Baire space of rough planar domains}\label{sec:framework}

Throughout the paper, eigenvalues are repeated according to multiplicity. For every nonempty open set $\Omega\subset D$, the Dirichlet Laplacian is the selfadjoint operator associated with the quadratic form
\[
 u\longmapsto\int_\Omega|\nabla u|^2,
 \qquad u\in H_0^1(\Omega).
\]
Since $\Omega\subset D$ and $D$ is bounded, its resolvent is compact.

The space of nonempty compact subsets of $\overline D$ is compact for the Hausdorff distance. Moreover, a Hausdorff limit of compact sets having at most $\ell$ connected components still has at most $\ell$ connected components. It follows that $\OO_\ell(D)$ is a compact metric space for \eqref{eq:dhc}, see \cite[Section~2.2]{HenrotPierre2018}. 

The main analytic input is the following result of \v{S}ver\'ak, see \cite{Sverak1993} or \cite[Theorem~3.4.14]{HenrotPierre2018}.

\begin{proposition}[\v{S}ver\'ak]\label{prop:sverak}
Let $(\Omega_n)_n\subset\OO_\ell(D)$ and $\Omega\in\OO_\ell(D)\setminus\{\varnothing\}$. If
\[
 \dhc(\Omega_n,\Omega)\longrightarrow0,
\]
then $\Omega_n$ $\gamma$-converges to $\Omega$. In particular, for every $j\geq1$,
\[
 \lambda_j(\Omega_n)\longrightarrow\lambda_j(\Omega).
\]
\end{proposition}

The first conclusion is the convergence in $H_0^1(D)$ of the solutions of the Dirichlet problems, extended by zero outside their domains. The eigenvalue convergence follows from the corresponding compact resolvent convergence. The implication from complementary-Hausdorff convergence to $\gamma$-convergence is the specifically planar part of the result.

We now isolate the topological point needed for the category argument.

\begin{proposition}\label{prop:Gdelta}
The space $\CC_\ell(D)$ is a $G_\delta$ subset of $\OO_\ell(D)$. Consequently, it is completely metrizable, and hence it is a Baire space.
\end{proposition}

\begin{proof}
Let $Q:=D\cap\Q^2$. For $p,q\in Q$, define
\[
 \mathcal N_p:=\{\Omega\in\OO_\ell(D):p\notin\Omega\}
\]
and
\[
 \begin{aligned}
 \mathcal P_{p,q}:=\{\Omega\in\OO_\ell(D):\;&
 p,q\in\Omega\text{ and}\\[-2pt]
 &p,q\text{ belong to the same connected component of }\Omega\}.
 \end{aligned}
\]
The set $\mathcal N_p$ is closed. Indeed, if $p\in K_{\Omega_n}$ for every $n$ and $K_{\Omega_n}\to K_\Omega$ in Hausdorff distance, then $p\in K_\Omega$.

The set $\mathcal P_{p,q}$ is open. If $\Omega\in\mathcal P_{p,q}$, the two points can be joined by a polygonal path $\Gamma\subset\Omega$. Since $\Gamma$ is compact,
\[
 \delta:=\dist(\Gamma,K_\Omega)>0.
\]
Whenever $\dhc(\Omega',\Omega)<\delta/2$, the path $\Gamma$ is contained in $\Omega'$, so $\Omega'\in\mathcal P_{p,q}$.

We claim that
\begin{equation}\label{eq:C-rational}
 \CC_\ell(D)=
 \bigl(\OO_\ell(D)\setminus\{\varnothing\}\bigr)
 \cap\bigcap_{p,q\in Q}
 \bigl(\mathcal N_p\cup\mathcal N_q\cup\mathcal P_{p,q}\bigr).
\end{equation}
If $\Omega$ is connected, each factor on the right is clearly satisfied. Conversely, if a nonempty open set $\Omega$ is disconnected, two of its connected components contain points $p,q\in Q$. For this pair, $p,q\in\Omega$ but $\Omega\notin\mathcal P_{p,q}$, and the corresponding factor fails.

In a metric space, every closed set is a $G_\delta$, open sets are also $G_\delta$, and finite unions of $G_\delta$ sets are $G_\delta$. Thus each set
\[
 \mathcal N_p\cup\mathcal N_q\cup\mathcal P_{p,q}
\]
is a $G_\delta$. The complement of the singleton $\{\varnothing\}$ is open, and the intersection in \eqref{eq:C-rational} is countable. This proves the first assertion. Finally, a $G_\delta$ subset of a compact metric space is completely metrizable, hence is a Baire space.
\end{proof}

\begin{remark}
\textcolor{black}{The metric $\dhc$ need not itself be complete on $\CC_\ell(D)$. Proposition~\ref{prop:Gdelta} asserts complete metrizability, which is the relevant intrinsic property. For instance, $(0,+\infty)\subset\R$ is not complete for the Euclidean distance, but the same topology is induced by the complete metric $d(x,y)=|\log(x/y)|$.}
\end{remark}

We next prove that the lack of boundary regularity in $\CC_\ell(D)$ does not prevent approximation by smooth domains. The preservation of the number of complementary components is included in the construction. 
We denote by $\cc(X)$ the set of connected components of a topological space $X$ and, when this set is finite, by $\#\cc(X)$ its cardinality.

\begin{proposition}[Smooth interior approximation]\label{prop:smooth-density}
Let $\Omega\in\CC_\ell(D)$. For every sufficiently small $\varepsilon>0$, there exists a connected open set $\Omega_\varepsilon$ with $C^\infty$ boundary such that
\begin{equation}\label{eq:smooth-approx}
 \begin{gathered}
  \Omega_\varepsilon\Subset\Omega,
  \qquad
  \#\cc(\overline D\setminus\Omega_\varepsilon)
  \leq \#\cc(\overline D\setminus\Omega),\\
  \Omega\setminus\Omega_\varepsilon
  \subset\{x\in\Omega:\dist(x,\partial\Omega)<\varepsilon\},
  \qquad
  \dhc(\Omega_\varepsilon,\Omega)\leq\varepsilon.
 \end{gathered}
\end{equation}
In particular, smooth domains are dense in $\CC_\ell(D)$.
\end{proposition}

\begin{proof}
Set $K:=\overline D\setminus\Omega$ and consider the compact set
\[
 A_\varepsilon:=
 \{x\in\overline D:\dist(x,K)\geq\varepsilon\}
 \Subset\Omega.
\]
It need not be connected, since passages of width less than $2\varepsilon$ may disappear. We first place all of it inside a compact connected subset of $\Omega$. Cover $A_\varepsilon$ by finitely many closed balls contained in $\Omega$. Since a connected open subset of $\R^2$ is polygonally connected, the centers of these balls can be joined to a fixed point of $\Omega$ by finitely many polygonal paths contained in $\Omega$. The union of the balls and paths is a compact connected set $L_\varepsilon$ satisfying
\[
 A_\varepsilon\subset L_\varepsilon\Subset\Omega.
\]

Choose $f\in C_c^\infty(\Omega)$, with $0\leq f\leq1$, which is equal to $1$ in a neighborhood of $L_\varepsilon$. Let $t\in(0,1)$ be a regular value of $f$, and let $U_\varepsilon$ be the connected component of $\{f>t\}$ containing $L_\varepsilon$. Then
\begin{equation}\label{eq:Ueps}
 A_\varepsilon\subset U_\varepsilon\Subset\Omega,
\end{equation}
and $U_\varepsilon$ is connected with $C^\infty$ boundary.

{The set $U_\varepsilon$ may have acquired artificial holes. We first
observe that the compact set $\overline D\setminus U_\varepsilon$ has only
finitely many connected components. Indeed, the components meeting $K$ are at
most $\#\cc(K)$ in number. A component disjoint from $K$ is compactly
contained in $\Omega$ and is a bounded complementary component of the smooth
domain $U_\varepsilon$. There are only finitely many such components because
$\partial U_\varepsilon$ is a compact smooth one-dimensional manifold and
therefore has finitely many connected components. Let
$C_1,\ldots,C_m$ denote all the connected components of
$\overline D\setminus U_\varepsilon$.}
{Those disjoint
from $K$ are precisely the artificial holes. Retain the components which meet
the original complement:}
\[
 I:=\{i\in\{1,\ldots,m\}:C_i\cap K\neq\varnothing\},
 \qquad
 F_\varepsilon:=\bigcup_{i\in I}C_i,
\]
and define
\[
 \Omega_\varepsilon:=\overline D\setminus F_\varepsilon.
\]
{Since $K\subset F_\varepsilon$ and $\partial D\subset K$, the set
$\Omega_\varepsilon$ is an open subset of $D$ contained in $\Omega$. It is
obtained from $U_\varepsilon$ by filling precisely the artificial holes. It is
therefore connected, and its boundary consists exactly of those connected
components of $\partial U_\varepsilon$ adjacent to retained complementary
components. In particular, it is smooth. The union of the finitely many
filled components is compact and disjoint from the compact set $K$, hence lies
at positive distance from $K$. Together with \eqref{eq:Ueps}, this gives
$\Omega_\varepsilon\Subset\Omega$.}

\textcolor{black}{Every connected component of $K$ is contained in a unique $C_i$, and every retained $C_i$ contains at least one such component. Hence the map sending a component of $K$ to the $C_i$ containing it is onto the set of retained components. Consequently,}
\[
 \#\cc(F_\varepsilon)\leq\#\cc(K)\leq\ell.
\]
Thus $\Omega_\varepsilon\in\CC_\ell(D)$.

Finally, $K\subset F_\varepsilon$ and, by \eqref{eq:Ueps}, every point $x\in F_\varepsilon\setminus K$ satisfies $\dist(x,K)<\varepsilon$. Since $F_\varepsilon\setminus K=\Omega\setminus\Omega_\varepsilon$ and $\dist(x,K)=\dist(x,\partial\Omega)$ for $x\in\Omega$, this proves the boundary-layer inclusion in \eqref{eq:smooth-approx}. Moreover,
\[
 d_H(F_\varepsilon,K) \leq\varepsilon,
\]
which is the last assertion in \eqref{eq:smooth-approx}.
\end{proof}

We can now transfer the classical generic-simplicity property from smooth domain spaces to $\CC_\ell(D)$.

\begin{proposition}[Generic simplicity]\label{prop:generic-simplicity}
The set $\mathcal S_\ell$ defined in \eqref{eq:simple-set} is residual in $\CC_\ell(D)$.
\end{proposition}

\begin{proof}
For $j\geq1$, set
\[
 \mathcal S_{\ell,j}:=
 \{\Omega\in\CC_\ell(D):
 \lambda_j(\Omega)<\lambda_{j+1}(\Omega)\}.
\]
Proposition~\ref{prop:sverak} shows that $\mathcal S_{\ell,j}$ is open.

We claim that domains with simple spectrum are dense in $\CC_\ell(D)$. Let $\Omega\in\CC_\ell(D)$ and let $\eta>0$. By Proposition~\ref{prop:smooth-density}, there exists a connected smooth domain $U\Subset D$ such that
\[
 U\in\CC_\ell(D),
 \qquad
 \dhc(U,\Omega)<\frac\eta2.
\]
\textcolor{black}{Micheletti's classical generic-simplicity theorem for smooth domains \cite{Micheletti1972,MichelettiMetric1972} gives a smooth domain $\widetilde U$ with simple Dirichlet spectrum, arbitrarily close to $U$ in a $C^m$ domain topology, for any fixed $m\geq3$. We may choose the perturbation so small that $\widetilde U\Subset D$, that $\widetilde U$ is diffeomorphic to $U$, and that}
\[
 \dhc(\widetilde U,U)<\frac\eta2.
\]
In particular, $\widetilde U$ is connected, has the same number of complementary components as $U$, and belongs to $\CC_\ell(D)$. Hence
\[
 \dhc(\widetilde U,\Omega)<\eta,
\]
which proves the claim.

The dense set of domains with simple spectrum is contained in every $\mathcal S_{\ell,j}$. Thus each $\mathcal S_{\ell,j}$ is open and dense. Since
\[
 \mathcal S_\ell=\bigcap_{j\geq1}\mathcal S_{\ell,j},
\]
the conclusion follows from Proposition~\ref{prop:Gdelta} and the Baire category theorem.
\end{proof}

\begin{proof}[Proof of Theorem~\ref{thm:framework-intro}]
Combine Propositions~\ref{prop:Gdelta}, \ref{prop:smooth-density} and \ref{prop:generic-simplicity}.
\end{proof}

\begin{remark}[Transfer of generic properties]\label{rem:transfer}
\textcolor{black}{The preceding argument applies to more than spectral simplicity, provided that the property under consideration is open in complementary-Hausdorff convergence at each fixed index. As an example, the author proved in the Micheletti topology that, generically among $C^m$ domains (with $m\geq 3$), every Dirichlet eigenfunction has nonzero average, and deduced that a generic domain has no nontrivial Euclidean symmetry \cite{Boulard2026GenericMeans}. 
These conclusions transfer to $\CC_\ell(D)$. Indeed, for each $j$, consider the set of domains for which $\lambda_j$ is simple and an associated $L^2$-normalized eigenfunction $\varphi_j$ satisfies $\int_\Omega\varphi_j\neq0$. 
This condition is independent of the sign of $\varphi_j$. Compact-resolvent convergence makes it open, since normalized eigenfunctions associated with a simple eigenvalue converge in $L^2$ up to sign. 
Proposition~\ref{prop:smooth-density} and the smooth-domain result of \cite{Boulard2026GenericMeans} make it dense. 
Taking the intersection over $j$ gives the claimed residual property, and the conclusion concerning symmetries follows as in that work. }
\end{remark}

\section{The local double-lock surgery}\label{sec:surgery}

We begin with a planar topological observation which will make the preservation of connectedness transparent.

\begin{lemma}[One-point attachment]\label{lem:attachment}
Let $U\subset\R^2$ be a connected open set, and identify $\R^2\cup\{\infty\}$ with $\mathbb S^2$. Let $F\subset\overline U$ be a compact set such that
\[
 F\cap(\mathbb S^2\setminus U)=\{y\}
\]
and assume that $F$ admits a strong deformation retraction onto $y$. Then $U\setminus F$ is connected.
\end{lemma}

\begin{proof}
{Set $K:=\mathbb S^2\setminus U$. Glue the strong deformation
retraction of $F$ onto $y$ to the identity on $K$. Since $F$ and $K$ are
closed and meet only at $y$, the pasting lemma gives a deformation retraction
of $K\cup F$ onto $K$. Alexander duality for compact subsets of
$\mathbb S^2$, as presented by Bredon \cite[Chapter~VI]{Bredon1993}, first
gives}
\[
 \widetilde H_0\bigl(\mathbb S^2\setminus(K\cup F)\bigr)
 \simeq \check H^1(K\cup F),
\]
where $\widetilde H_{*}$ denotes the reduced singular homology with coefficients in $\mathbb Z$, and $\check H^*$ the \v{C}ech cohomology with coefficients in $\mathbb Z$. Recall that $\widetilde H_0(X)=0$ if and only if the topological space $X$ is path connected.

Now, since $F$ is attached to $K$ at the single point $y$ and strongly retracts onto that point, by homotopy invariance we have
\[
 \check H^1(K\cup F)
 \simeq \check H^1(K).
\]
\textcolor{black}{Finally, applying Alexander duality once more gives $\check H^1(K)
 \simeq \widetilde H_0(U)=0$, because $U$ is path connected. Since}
\[
\mathbb S^2\setminus(K\cup F)
 =(\mathbb S^2\setminus K)\setminus F
 =U\setminus F,
\]
we obtain $\widetilde H_0(U\setminus F)=0$, and hence $U\setminus F$ is connected.
\end{proof}

{Finite embedded trees satisfy the hypothesis of
Lemma~\ref{lem:attachment}: contract their terminal edges successively toward
the root $y$. The same contraction works when pairwise disjoint closed disks
are attached at terminal vertices, after each disk is first contracted to its
attachment point. We will use both forms.}
\textcolor{black}{Here an embedded arc is the image of an injective continuous map $[0,1]\to\R^2$, and an embedded finite tree is a finite connected graph without cycles whose edges are embedded arcs meeting only at their prescribed common vertices.}

\textcolor{black}{The disk $B_*$ in the next lemma is kept untouched by the surgery. Besides guaranteeing nonemptiness, it will provide in Proposition~\ref{prop:tuning} the uniform bound $\lambda_N(M_{r,s})\leq\lambda_N(B_*)$.}

\begin{lemma}[Geometric implantation]\label{lem:geometry}
Let $\Omega\in\CC_\ell(D)$, let $\eta>0$, and let $B_*\Subset\Omega$ be an open disk. There exist two points $c_1,c_2\in\Omega$, a number $r_0>0$, and, for every $r,s\in(0,r_0)$ and $\varepsilon>0$ sufficiently small, open sets
\[
 M_{r,s},\qquad A_{r,s},\qquad \Omega_{r,s,\varepsilon}
\]
with the following properties:
\begin{enumerate}[label=\textup{(\roman*)}]
 \item $M_{r,s}\in\CC_\ell(D)$, $B_*\subset M_{r,s}$, and
 \[
 A_{r,s}=M_{r,s}\,\sqcup\,B(c_1,r)\,\sqcup\,B(c_2,s)
 \in\OO_\ell(D).
 \]
 \item $\Omega_{r,s,\varepsilon}\in\CC_\ell(D)$ and
 \[
 \dhc(\Omega_{r,s,\varepsilon},A_{r,s})\longrightarrow0
 \qquad(\varepsilon\to0).
 \]
 \item Every set in the construction lies in the $\eta$-ball about $\Omega$:
 \[
 \dhc(\Omega_{r,s,\varepsilon},\Omega)<\eta,
 \qquad
 \dhc(A_{r,s},\Omega)<\eta.
 \]
 \item The map $(r,s)\mapsto M_{r,s}$ is continuous for $\dhc$. Moreover, if $0<t_1<t_2<r_0$, then
 \[
 M_{t_2,t_2}\subset M_{t_1,t_1}.
 \]
\end{enumerate}
\end{lemma}

\begin{proof}
Write $K:=K_\Omega$. Choose $x\in\Omega\setminus\overline{B_*}$ so close to $K$ that
\[
 d:=\dist(x,K)<\eta/2,
\]
\textcolor{black}{and let $y\in K$ realize this distance. The open segment $(y,x]$ lies in $\Omega$: otherwise it would meet $K$ at a point closer to $x$ than $y$.
Choose a ball centered at $x$, contained in $\Omega\setminus\overline{B_*}$, and of radius smaller than $\eta-d$. Inside this ball choose two points $c_1,c_2$. Join them to the segment $[y,x]$ by two disjoint branches meeting the segment at a common branching point. After decreasing the ball, all these arcs are embedded, their only intersections are the prescribed vertices, and the whole device lies in the $\eta$-neighborhood of $K$. See Figure~\ref{fig:locks} for the final construction.}

For $r,s<r_0$ small, let $C_1(r)=B(c_1,r)$ and $C_2(s)=B(c_2,s)$. Each branch is chosen radial near its center and stops at the boundary of the corresponding disk. Denote by $L_{r,s}$ the union of the main stem, the two branches, and the closed disks $\overline{C_1(r)}$, $\overline{C_2(s)}$. Thus $L_{r,s}$ is a tree with two terminal vertices thickened into disks, it meets $\mathbb S^2\setminus\Omega$ only at $y$, and it retracts onto $y$. Set
\[
 M_{r,s}:=\Omega\setminus L_{r,s}.
\]
\textcolor{black}{Lemma~\ref{lem:attachment} shows that $M_{r,s}$ is connected. The construction is disjoint from $B_*$, and the new part of the complement is connected and meets $K$. Hence it creates no new connected component of the complement $\overline D\setminus M_{r,s}$. Therefore $M_{r,s}\in\CC_\ell(D)$.}

Replace now each closed disk in $L_{r,s}$ by its boundary circle, while retaining the stems and branches, and denote the resulting compact set by $T_{r,s,0}$. Then
\[
 \Omega\setminus T_{r,s,0}
 =M_{r,s}\,\sqcup\,C_1(r)\,\sqcup\,C_2(s).
\]
This is $A_{r,s}$. Again, $K\cup T_{r,s,0}$ has no more connected components than $K$, because the whole graph is attached to one component of $K$.

For $\varepsilon>0$, remove from each circular wall an open arc of length $\varepsilon$, away from the attachment point of the branch, and call the remaining compact set $T_{r,s,\varepsilon}$. Each broken circle is an arc, so $T_{r,s,\varepsilon}$ is a finite tree attached to $K$ only at $y$. Define
\[
 \Omega_{r,s,\varepsilon}:=\Omega\setminus T_{r,s,\varepsilon}.
\]
Lemma~\ref{lem:attachment} proves that $\Omega_{r,s,\varepsilon}$ is connected. Since $T_{r,s,\varepsilon}$ is connected to $K$, the number of connected components of the complement does not increase, and hence $\Omega_{r,s,\varepsilon}\in\CC_\ell(D)$. This is precisely where the open doors are used.

As $\varepsilon\to0$, the two broken circles converge in Hausdorff distance to the full circles. Thus $\Omega_{r,s,\varepsilon}\to A_{r,s}$ for $\dhc$. All the compact sets added to $K$ lie in its $\eta$-neighborhood, which proves (iii). The disks and the radial branches depend continuously on their radii in Hausdorff distance, proving the first part of (iv). Finally, when a radius grows, the radial branch becomes shorter, and the removed portion of the branch is contained in the larger disk. Consequently $L_{t_1,t_1}\subset L_{t_2,t_2}$ for $t_1<t_2$, which gives the last inclusion.
\end{proof}

\begin{figure}[ht]
\centering
\begin{tikzpicture}[scale=0.9,line cap=round,line join=round]
  \draw[thick]
    plot[smooth cycle, tension=0.72] coordinates{
      (-3.2,0) (-2.5,1.8) (-0.3,2.2) (2.1,2.05)
      (3.8,0.1) (2.3,-2.05) (-0.2,-2.3) (-2.6,-1.85)};
  \node at (-2.5,1.35) {$\Omega$};

  \coordinate (y)  at (-3.2,0);
  \coordinate (b)  at (-1.0,0);
  \coordinate (c1) at (0.65,0.75);
  \coordinate (c2) at (0.70,-0.75);
  \coordinate (e1) at (0.122,0.510);
  \coordinate (e2) at (0.224,-0.540);

  \draw[very thick] (y) -- (b);
  \draw[very thick] (b) -- (e1);
  \draw[very thick] (b) -- (e2);

  \draw[very thick]
    (c1) ++(44:0.58)
    arc[start angle=44,end angle=376,radius=0.58];

  \draw[very thick]
    (c2) ++(344:0.52)
    arc[start angle=344,end angle=676,radius=0.52];

  \fill (y)  circle (2.2pt);
  \fill (c1) circle (1.3pt);
  \fill (c2) circle (1.3pt);

  \node[below right] at (y) {$\,y\in K_\Omega$};
  \node at (0.65,0.54)  {$c_1$};
  \node at (0.70,-0.96) {$c_2$};
  \node at (0.65,1.03)  {$C_1$};
  \node at (0.70,-0.51) {$C_2$};

  \node[align=center,font=\small] at (1.95,-0.05)
    {open doors\\[-1pt]$\varepsilon>0$};
\end{tikzpicture}
\caption{Schematic drawing of the geometric device used in the proof of {Lemma~\ref{lem:geometry}}.}
\label{fig:locks}
\end{figure}
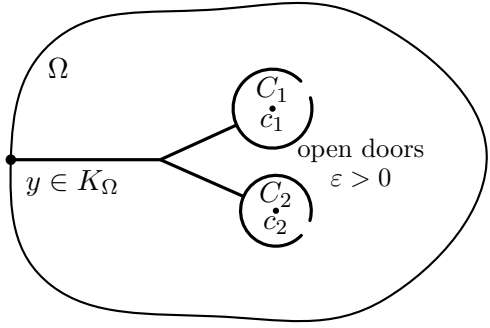

{The freedom in the two radii now allows us to create, without imposing
global simplicity, a positive gap between consecutive distinct eigenvalues at
an arbitrarily high index.}

\begin{proposition}[Spectral tuning]\label{prop:tuning}
Let $\Omega\in\CC_\ell(D)$, let $k,N\geq1$, and let $\eta>0$. There exist $j\geq N$ and $\widetilde\Omega\in\CC_\ell(D)$ such that
\[
 \dhc(\widetilde\Omega,\Omega)<\eta
 \qquad\text{and}\qquad
 0<\lambda_{j+1}(\widetilde\Omega)-\lambda_j(\widetilde\Omega)<\frac1k.
\]
\end{proposition}

\begin{proof}
Choose a disk $B_*\Subset\Omega$ and apply Lemma~\ref{lem:geometry}. Write
\[
 M_t:=M_{t,t},
 \qquad
 \mu(t):=\lambda_1(B(0,t))=\frac{\lambda_1(B(0,1))}{t^2}.
\]
By Lemma~\ref{lem:geometry}(iv) and domain monotonicity, for each $n\geq1$ the function
\[
 t\longmapsto\lambda_n(M_t)
\]
is nondecreasing, whereas $t\mapsto\mu(t)$ is strictly decreasing. Consequently, for every $n$ the equation
\[
 \mu(t)=\lambda_n(M_t)
\]
has at most one solution. The set of radii $t$ for which $\mu(t)\in\Spec(-\Delta_{M_t})$ is therefore at most countable.

Moreover, $B_*\subset M_t$ for every $t$, so domain monotonicity gives
\[
 \lambda_N(M_t)\leq\lambda_N(B_*).
\]
Since $\mu(t)\to+\infty$ as $t\to0$, we may choose $t\in(0,r_0)$ outside the countable exceptional set and so small that
\begin{equation}\label{eq:mu-above-N}
 \mu(t)>\lambda_N(M_t).
\end{equation}

{Since $\mu(t)\notin\Spec(-\Delta_{M_t})$, discreteness and
\eqref{eq:mu-above-N} allow us to choose a bounded open interval
$J\ni\mu(t)$ so small that
\[
 \overline J\cap\Spec(-\Delta_{M_t})=\varnothing,\qquad
 \lambda_N(M_t)<\inf J,\qquad
 \sup J<\lambda_2(B(0,t)).
\]
Take $m$ such that $\lambda_{m+1}(D)>\sup J$. Since $M_{t,s}\subset D$,
domain monotonicity gives
\[
 \lambda_{m+1}(M_{t,s})\geq\lambda_{m+1}(D)>\sup J,
\]
so only its first $m$ eigenvalues can meet $J$. Their continuity at $s=t$,
given by Proposition~\ref{prop:sverak} and Lemma~\ref{lem:geometry}(iv),
together with scaling continuity for the disks, allows us to choose
$s\neq t$ arbitrarily close to $t$ such that, writing $I$ for the closed
interval with endpoints $\mu(t)$ and $\mu(s)$,}
\begin{equation}\label{eq:tuning-properties}
 \begin{aligned}
  &0<|\mu(s)-\mu(t)|<\frac{1}{3k},\\
  &I\cap\Spec(-\Delta_{M_{t,s}})=\varnothing,\\
  &\lambda_N(M_{t,s})<\min I,\\
  &\max I<\min\{\lambda_2(B(0,t)),\lambda_2(B(0,s))\}.
 \end{aligned}
\end{equation}
The last condition follows from the strict inequality $\lambda_2(B(0,t))>\lambda_1(B(0,t))$ and continuity under scaling.

The spectrum of the disjoint union
\[
 A_{t,s}=M_{t,s}\,\sqcup\,B(c_1,t)\,\sqcup\,B(c_2,s)
\]
is the union, with multiplicities, of the spectra of its three components. It follows from \eqref{eq:tuning-properties} that $\mu(t)$ and $\mu(s)$ are two simple consecutive eigenvalues of $A_{t,s}$. 
By \eqref{eq:mu-above-N} and \eqref{eq:tuning-properties}, there are at least $N$ eigenvalues below them. Hence there exists an index $j\geq N$ such that
\[
 \{\lambda_j(A_{t,s}),\lambda_{j+1}(A_{t,s})\}
 =\{\mu(t),\mu(s)\}.
\]

Finally, $\Omega_{t,s,\varepsilon}\to A_{t,s}$ in complementary-Hausdorff distance as $\varepsilon\to0$. By Proposition~\ref{prop:sverak},
\[
 \lambda_j(\Omega_{t,s,\varepsilon})\longrightarrow\lambda_j(A_{t,s}),
 \qquad
 \lambda_{j+1}(\Omega_{t,s,\varepsilon})\longrightarrow\lambda_{j+1}(A_{t,s}).
\]
The limiting gap is positive and smaller than $1/(3k)$. For $\varepsilon>0$ sufficiently small, the corresponding gap on $\Omega_{t,s,\varepsilon}$ therefore belongs to $(0,1/k)$. Lemma~\ref{lem:geometry} also gives
\[
 \Omega_{t,s,\varepsilon}\in\CC_\ell(D),
 \qquad
 \dhc(\Omega_{t,s,\varepsilon},\Omega)<\eta.
\]
Taking $\widetilde\Omega=\Omega_{t,s,\varepsilon}$ concludes the proof.
\end{proof}

\section{{Generic failure of uniform separation}}\label{sec:baire}

For $k,N\geq1$, define
\begin{equation}\label{eq:UkN}
 \mathcal U_{k,N}:=
 \bigcup_{j\geq N}
 \left\{\Omega\in\CC_\ell(D):
 0<\lambda_{j+1}(\Omega)-\lambda_j(\Omega)<\frac1k\right\}.
\end{equation}

{We first relate these sets, which are expressed using eigenvalues counted
with multiplicity, to the distinct spectrum. For a fixed $\Omega$, let}
\[
 {p_m:=\max\{j\geq1:\lambda_j(\Omega)=\nu_m(\Omega)\}.}
\]
{Every eigenvalue has finite multiplicity, so $p_m\to\infty$, and}
\begin{equation}\label{eq:block-boundary}
 {\nu_{m+1}(\Omega)-\nu_m(\Omega)
 =\lambda_{p_m+1}(\Omega)-\lambda_{p_m}(\Omega)>0.}
\end{equation}
{Conversely, every positive difference
$\lambda_{j+1}(\Omega)-\lambda_j(\Omega)$ occurs at one of the indices
$j=p_m$ and equals the corresponding difference in
\eqref{eq:block-boundary}. Consequently,}
\begin{equation}\label{eq:G-distinct}
 {\bigcap_{k\geq1}\bigcap_{N\geq1}\mathcal U_{k,N}
 =
 \left\{\Omega\in\CC_\ell(D):
 \inf_{m\geq1}
 \bigl(\nu_{m+1}(\Omega)-\nu_m(\Omega)\bigr)=0\right\}.}
\end{equation}

\begin{proof}[Proof of Theorem~\ref{thm:main}]
For every fixed $j$, Proposition~\ref{prop:sverak} shows that the maps
\[
 \Omega\longmapsto\lambda_j(\Omega),
 \qquad
 \Omega\longmapsto\lambda_{j+1}(\Omega)
\]
are continuous on $\CC_\ell(D)$. Both inequalities in \eqref{eq:UkN} are strict, so $\mathcal U_{k,N}$ is open. Proposition~\ref{prop:tuning} shows that it is dense.

By Proposition~\ref{prop:Gdelta}, $\CC_\ell(D)$ is a Baire space. Hence
\[
 {\bigcap_{k\geq1}\bigcap_{N\geq1}\mathcal U_{k,N}}
\]
{is a dense $G_\delta$ subset of $\CC_\ell(D)$. The identity
\eqref{eq:G-distinct} proves the theorem.}
\end{proof}

\begin{proof}[Proof of Corollary~\ref{cor:simple-refinement}]
{Theorem~\ref{thm:main} shows that $\mathcal G_\ell$ is residual, while
Proposition~\ref{prop:generic-simplicity} shows that $\mathcal S_\ell$ is
residual. Hence $\mathcal G_\ell\cap\mathcal S_\ell$ is residual. If
$\Omega$ belongs to this intersection, then
$\nu_m(\Omega)=\lambda_m(\Omega)$ for every $m$, and therefore}
\[
 {\liminf_{j\to\infty}
 \bigl(\lambda_{j+1}(\Omega)-\lambda_j(\Omega)\bigr)=0.}
\]
{Since every gap is positive, its infimum is therefore zero.}
\end{proof}

\section{{The exceptional uniform-gap problem}}\label{sec:questions}

{The preceding results leave a clean existence problem. In every space
$\CC_\ell(D)$, uniform separation of the distinct Dirichlet spectrum is
exceptional, while a generic domain has simple spectrum and arbitrarily
small positive consecutive gaps. The complementary-Hausdorff framework thus
isolates uniform separation as a genuinely nongeneric rigidity phenomenon.}

\begin{question}[The planar uniform-gap problem]\label{q:simple}
{Does there exist a bounded planar domain $\Omega\subset\R^2$ whose
Dirichlet eigenvalues, repeated according to multiplicity, satisfy}
\[
 {\inf_{j\geq1}
 \bigl(\lambda_{j+1}(\Omega)-\lambda_j(\Omega)\bigr)>0?}
\]
{Equivalently, does there exist a bounded planar domain with simple,
uniformly separated Dirichlet spectrum?}
\end{question}

{The explicitly solvable integrable polygons do not answer
Question~\ref{q:simple}. For rectangles with rational squared aspect ratio,
as well as for the equilateral, right isosceles and hemi-equilateral
triangles, the distinct eigenvalues lie in an arithmetic lattice and are
uniformly separated, but the complete spectrum has arithmetic
multiplicities \cite{MardbyRowlett2025,BerardHelffer2016}. Together with
rectangles, these three triangles exhaust the integrable polygons
\cite{MardbyRowlett2025}. Irrational rectangles have simple spectrum, but
their gaps have zero lower limit by
\cite[Proposition~2.2]{BlomerBourgainRadziwillRudnick2016}. Thus every
standard explicit candidate loses one of the two required properties.}

{Both standard conjectural models of spectral statistics likewise point
toward a negative answer. For generic quantum-integrable billiards, the
Poisson statistics predicted by Berry and Tabor imply arbitrarily small gaps
\cite{BerryTabor1977}. For chaotic billiards, the
Bohigas--Giannoni--Schmit conjecture predicts random-matrix statistics
\cite{BohigasGiannoniSchmit1984}. Although level repulsion makes very small
gaps less frequent, the predicted spacing distributions assign positive
probability to every interval $(0,\varepsilon)$. The smallest gap among the
first $N$ levels is therefore still expected to tend to zero as
$N\to\infty$.}

{The present paper turns this evidence into a rigorous Baire-category
statement. For every $D$ and $\ell$, any positive solution of
Question~\ref{q:simple} lying in $\CC_\ell(D)$ belongs to a meagre set.
Determining whether a bounded planar domain with simple, uniformly separated
Dirichlet spectrum exists at all remains a natural rigidity problem in spectral geometry.}

\section*{Acknowledgements}
The author is grateful to Jean-Michel Coron for suggesting the question, and to Emmanuel Tr\'elat for helpful discussions.
This work was supported by the ANR-Tremplin StarPDE (ANR-24-ERCS-0010).

\bibliographystyle{abbrv}
\bibliography{spectral_gap_references}

\end{document}